\documentclass[leqno,11pt]{article}
\usepackage[utf8]{inputenc}
\usepackage[T1]{fontenc}
\usepackage{microtype}

\usepackage[letterpaper]{geometry}

\ifdefined\screenview
  \edef\mtht{\the\textheight}
  \edef\mtwd{\the\textwidth}
  \usepackage{xcolor}
  \definecolor{BackgroundColor}{RGB}{253, 246, 227}
  \pagecolor{BackgroundColor}
\fi

\usepackage{amsmath}
\usepackage{amsthm}
\usepackage{tikz-cd}
\usetikzlibrary{arrows} 
\tikzset{
  commutative diagrams/.cd, 
  arrow style=tikz, 
  diagrams={>=stealth}
}
\usetikzlibrary{matrix,decorations.pathreplacing,calc}

\usepackage{textcomp}

\usepackage[sb]{libertine}
\usepackage[varqu,varl]{zi4}
\usepackage[libertine,bigdelims,vvarbb]{newtxmath} 
\usepackage[supstfm=libertinesups,%
  supscaled=1.2,%
  raised=-.13em]{superiors}
\useosf 
\usepackage[scr=boondox,cal=euler]{mathalfa}

\usepackage{slashed}
\usepackage{esint} 
\usepackage[english]{babel}

\usepackage{imakeidx}
\makeindex[intoc]

\usepackage{csquotes}
\usepackage[
  backend=biber,
  hyperref=true,
  backref=true,
  isbn=false,
  doi=true,
  natbib=true,
  eprint=true,
  useprefix=true,
  maxcitenames=99,
  maxbibnames=99,  
  maxalphanames=99, 
  minalphanames=99,
  safeinputenc,
  style=alphabetic,
  citestyle=alphabetic,
  block=space,
  datamodel=preamble/ext-eprint
]{biblatex}
\usepackage[
  hypertexnames = false,
  colorlinks    = true,
  citecolor     = gray,
  linkcolor     = gray,
  urlcolor      = gray,
  breaklinks
]{hyperref}
\makeatletter
\DeclareFieldFormat{arxiv}{%
  arXiv\addcolon\space
  \ifhyperref
    {\href{http://arxiv.org/\abx@arxivpath/#1}{%
       \nolinkurl{#1}%
       \iffieldundef{arxivclass}
         {}
         {\addspace\texttt{\mkbibbrackets{\thefield{arxivclass}}}}}}
    {\nolinkurl{#1}
     \iffieldundef{arxivclass}
       {}
       {\addspace\texttt{\mkbibbrackets{\thefield{arxivclass}}}}}}
\makeatother
\DeclareFieldFormat{mr}{%
  MR\addcolon\space
  \ifhyperref
    {\href{http://www.ams.org/mathscinet-getitem?mr=MR#1}{\nolinkurl{#1}}}
    {\nolinkurl{#1}}}
\DeclareFieldFormat{zbl}{%
  Zbl\addcolon\space
  \ifhyperref
    {\href{http://zbmath.org/?q=an:#1}{\nolinkurl{#1}}}
    {\nolinkurl{#1}}}
\renewbibmacro*{eprint}{%
  \printfield{arxiv}%
  \newunit\newblock
  \printfield{mr}%
  \newunit\newblock
  \printfield{zbl}%
  \newunit\newblock
  \iffieldundef{eprinttype}
    {\printfield{eprint}}
    {\printfield[eprint:\strfield{eprinttype}]{eprint}}
}
\AtEveryBibitem{%
  \clearlist{publisher}%
}
\AtEveryBibitem{%
  \clearlist{address}%
}
\DeclareFieldFormat[article,inproceedings,inbook,incollection,thesis]{title}{\textit{#1}}
\renewbibmacro{in:}{}
\newcommand{\printreferences}{\printbibliography[heading=bibintoc]}

\usepackage[inline,shortlabels]{enumitem}

\usepackage{subcaption}

\usepackage[yyyymmdd]{datetime}

\usepackage{etoolbox}
\ifundef{\abstract}{}{\patchcmd{\abstract}%
    {\quotation}{\quotation\noindent\ignorespaces}{}{}}

\usepackage[super]{nth}

\usepackage{aliascnt}

\numberwithin{equation}{section}

\renewcommand{\eqref}[1]{\hyperref[#1]{\rm(\ref*{#1})}}

\def\makeautorefname#1#2{\AtBeginDocument{\expandafter\def\csname#1autorefname\endcsname{#2}}}

\newcommand{\mynewtheorem}[2]{
  \newaliascnt{#1}{equation}          
  \newtheorem{#1}[#1]{#2}
  \aliascntresetthe{#1}
  \makeautorefname{#1}{#2}
}

\mynewtheorem{axiom}{Axiom}
\newtheorem*{axiom*}{Axiom}
\mynewtheorem{theorem}{Theorem}
\newtheorem*{theorem*}{Theorem}
\mynewtheorem{prop}{Proposition}
\newtheorem*{prop*}{Proposition}

\mynewtheorem{cor}{Corollary}
\mynewtheorem{construction}{Construction}
\mynewtheorem{lemma}{Lemma}
\mynewtheorem{conjecture}{Conjecture}
\newtheorem*{conjecture*}{Conjecture}
\mynewtheorem{hyp}{Hypothesis}

\numberwithin{substep}{step}
\makeautorefname{step}{Step}
\makeautorefname{substep}{Step}

\numberwithin{subcase}{case}
\makeautorefname{case}{Case}
\makeautorefname{subcase}{case}

\usepackage{etoolbox}
\AtBeginEnvironment{proof}{\setcounter{step}{0}}

\theoremstyle{remark}
\mynewtheorem{remark}{Remark}
\newtheorem*{remark*}{Remark}
\mynewtheorem{convention}{Convention}
\newtheorem*{convention*}{Convention}
\newtheorem*{conventions*}{Conventions}

\theoremstyle{remark}
\mynewtheorem{warning}{Warning}

\theoremstyle{definition}
\mynewtheorem{definition}{Definition}
\newtheorem*{definition*}{Definition}
\mynewtheorem{notation}{Notation}
\mynewtheorem{data}{Data}
\mynewtheorem{example}{Example}
\newtheorem*{example*}{Example}
\mynewtheorem{exercise}{Exercise}
\mynewtheorem{solution}{Solution}
\mynewtheorem{question}{Question}
\mynewtheorem{problem}{Problem}
\newtheorem*{question*}{Question}

\mynewtheorem{summary}{Summary}

\makeautorefname{table}{Table}        
\makeautorefname{chapter}{Chapter}
\makeautorefname{section}{Section}
\makeautorefname{subsection}{Section}
\makeautorefname{subsubsection}{Section}
\makeautorefname{footnote}{Footnote}
\AtBeginDocument{\def\itemautorefname~#1\null{(#1)\null}}
\AtBeginDocument{\def\equationautorefname~#1\null{(#1)\null}}

\let\C\undefined
\let\U\undefined

\usepackage{bm}
\usepackage{mathtools} 
\usepackage{stmaryrd} 

\DeclareFontFamily{U}{mathx}{\hyphenchar\font45}
\DeclareFontShape{U}{mathx}{m}{n}{
      <5> <6> <7> <8> <9> <10>
      <10.95> <12> <14.4> <17.28> <20.74> <24.88>
      mathx10
      }{}
\DeclareSymbolFont{mathx}{U}{mathx}{m}{n}
\DeclareFontSubstitution{U}{mathx}{m}{n}
\DeclareMathAccent{\widecheck}{0}{mathx}{"71}
\DeclareMathAccent{\wideparen}{0}{mathx}{"75}

\DeclareMathOperator{\Amp}{Amp}

\DeclareMathOperator{\Gr}{Gr}

\DeclareMathOperator{\HF}{\HF}

\DeclareMathOperator{\Pic}{Pic}

\DeclareMathOperator{\im}{im}

\DeclareMathOperator{\res}{res}
\DeclareMathOperator{\rk}{rk}

\DeclarePairedDelimiter{\abs}{\lvert}{\rvert}

\DeclarePairedDelimiter{\set}{\lbrace}{\rbrace}
\def\({\left(}
\def\){\right)}
\def\<{\left\langle}
\def\>{\right\rangle}

\newcommand{\C}{{\mathbf{C}}}
\newcommand{\Gtwo}{G_2}

\newcommand{\N}{{\mathbf{N}}}

\newcommand{\PU}{{\P\U}}

\newcommand{\Q}{\mathbf{Q}}

\newcommand{\R}{\mathbf{R}}

\newcommand{\SO}{\mathrm{SO}}

\newcommand{\Sch}{\textbf{Sch}}
\newcommand{\Set}{\textbf{Set}}

\newcommand{\U}{\mathrm{U}}

\newcommand{\Z}{\mathbf{Z}}

\newcommand{\andq}{\text{and}\quad}

\newcommand{\co}{\mskip0.5mu\colon\thinspace}

\newcommand{\defined}[2][\key]{\def\key{#2}\textbf{#2}\index{#1}}

\newcommand{\ev}{\mathrm{ev}}

\newcommand{\iso}{\cong}

\newcommand{\qandq}{\quad\text{and}\quad}

\newcommand{\sEnd}{\mathrm{\sE nd}}

\renewcommand{\O}{\mathrm{O}}
\renewcommand{\P}{\mathbf{P}}

\renewcommand{\emptyset}{\varnothing}
\renewcommand{\epsilon}{\varepsilon}
\renewcommand{\setminus}{{\backslash}}

\renewcommand{\leq}{\leqslant}

\makeatletter
\renewcommand*\env@matrix[1][*\c@MaxMatrixCols c]{%
  \hskip -\arraycolsep
  \let\@ifnextchar\new@ifnextchar
  \array{#1}}

\renewcommand\xleftrightarrow[2][]{%
  \ext@arrow 9999{\longleftrightarrowfill@}{#1}{#2}}
\newcommand\longleftrightarrowfill@{%
  \arrowfill@\leftarrow\relbar\rightarrow}
\makeatother

\newcommand{\rd}{{\rm d}}

\newcommand{\bM}{{\mathbf{M}}}

\newcommand{\cZ}{\mathcal{Z}}

\newcommand{\sA}{\mathscr{A}}

\newcommand{\sE}{\mathscr{E}}
\newcommand{\sF}{\mathscr{F}}

\newcommand{\sL}{\mathscr{L}}

\newcommand{\sO}{\mathscr{O}}

\newcommand{\fg}{{\mathfrak g}}

\newcommand{\fr}{{\mathfrak r}}

\newcommand{\bomega}{{\bm{\omega{}}}}

\pdfoutput=1

\author{
  Thomas Walpuski
}

\title{$\Gtwo$--instantons over twisted connected sums: an example}
\date{2016-07-31}
\begin{document}

\maketitle

\begin{abstract}
  Using earlier work of Sá Earp and the author \cite{SaEarp2013} we construct an irreducible unobstructed $\Gtwo$--instanton on an $\SO(3)$--bundle over a twisted connected sum $\Gtwo$--manifold recently discovered by Crowley and Nordström \cite{Crowley2014}.
\end{abstract}

\paragraph{Changes to the published version}
A former version of this article has been published in \href{https://dx.doi.org/10.4310/MRL.2016.v23.n2.a11}{Mathematical Research Letters, Volume 23, Issue 2, pp. 529–544 (2016)}.
The present version is identical to the published article, except for aesthetic changes and the addition of a paragraph in \autoref{Sec_CNExample} justifying the application of \cite[Theorem 7.5]{Moishezon1967} in more detail.

\section{Introduction}
\label{Sec_Introduction}

In order to put this note into context and help the reader appreciate its significance, we (very) briefly recall some ideas from the study of gauge theory on $\Gtwo$--manifolds.

\begin{definition}
  A connection $A \in \sA(E)$ on a $G$--bundle $E$ over a $\Gtwo$--manifold $Y$ is called a \defined{$\Gtwo$--instanton} if its curvature satisfies
  \begin{equation}
    \label{Eq_G2Instanton}
    F_A \wedge \psi = 0
  \end{equation}
  with $\psi \coloneq *\phi$ and $\phi \in \Omega^3(Y)$ denoting the $\Gtwo$--structure on $Y$.
\end{definition}

In their visionary article \cite{Donaldson1998} Donaldson and Thomas speculated that ``counting'' $\Gtwo$--instantons might lead to an interesting enumerative invariant.
Although almost two decades have passed, it is still not understood what the precise definition of this invariant ought to be;
however, see Donaldson and Segal \cite{Donaldson2009}, the author \cites{Walpuski2013a,Walpuski2013}, and  Haydys and the author \cite{Haydys2014} for some recent progress.
What is clear, nonetheless, is that irreducible unobstructed $\Gtwo$--instantons should contribute with $\pm 1$ (depending on orientations).

\begin{definition}
  A $\Gtwo$--instanton $A \in \sA(E)$ is called \defined{irreducible} (\defined{unobstructed}) if the elliptic complex
  \begin{equation*}
    \Omega^0(Y,\fg_E) \xrightarrow{\rd_A}
    \Omega^1(Y,\fg_E) \xrightarrow{\psi\wedge\rd_A}
    \Omega^6(Y,\fg_E) \xrightarrow{\rd_A}
    \Omega^7(Y,\fg_E)
  \end{equation*}
  has vanishing cohomology in degree zero (one).
\end{definition}

In \cite{SaEarp2013} Sá Earp and the author developed a method for constructing irreducible unobstructed $\Gtwo$--instantons over twisted connected sums.
So far, however, we were unable to find a single instance of the input required for this construction.
This brief note is meant to ameliorate this disgraceful situation by showing that our method can be used to produce at least one example.

Let us briefly recall the twisted connected sum construction, a rich source of $\Gtwo$--manifolds, which was suggested by Donaldson, pioneered by Kovalev~\cite{Kovalev2003} and later extended and improved by Kovalev and Lee~\cite{Kovalev2011}, and Corti, Haskins, Nordström and Pacini~\cite{Corti2012a}.

\begin{definition}
  A \defined{building block} is a non-singular algebraic $3$--fold $Z$ together with a projective morphism $f\co Z \to \P^1$ such that:
  \begin{itemize}
  \item
    the anticanonical class $-K_Z \in H^2(Z)$ is primitive,
  \item
    $\Sigma \coloneq f^{-1}(\infty)$ is a smooth $K3$ surface and $\Sigma \sim -K_Z$.
  \end{itemize}
  A \defined{framing} of a building block $(Z,\Sigma)$ consists of a hyperkähler structure $\bomega = (\omega_I,\omega_J,\omega_K)$ on $\Sigma$ such that $\omega_J + i\omega_K$ is of type $(2,0)$ as well as a Kähler class on $Z$ whose restriction to $\Sigma$ is $[\omega_I]$.\footnote{%
    The existence of such a class is not guaranteed a priori.}
\end{definition}

For the purpose of this article we are mostly interested in the following class of building blocks introduced by Kovalev~\cite{Kovalev2003}. 

\begin{definition}
  A building block is said to be of \defined{Fano type} if it is obtained by blowing-up a Fano $3$--fold $W$ along the base locus of general pencil $|\Sigma_0,\Sigma_\infty| \subset |-K_W|$.
  (See \autoref{Sec_FanoTypeBuildingBlocks} for more details on this construction.)
\end{definition}

Given a framed building block $(Z,\Sigma,\bomega)$, using the work of Haskins, Hein and Nordström \cite{Haskins2012}, we can make $V \coloneq Z\setminus\Sigma$ into an asymptotically cylindrical (ACyl) Calabi--Yau $3$--fold with asymptotic cross-section $S^1 \times\Sigma$;
hence, $Y \coloneq S^1\times V$ is an ACyl $\Gtwo$--manifold with asymptotic cross-section $T^2\times\Sigma$.

\begin{definition}
  A \defined{matching} of pair of framed building blocks  $(Z_\pm,\Sigma_\pm,\bomega_\pm)$ is a hyperkähler rotation $\fr\co \Sigma_+ \to \Sigma_-$, i.e., a diffeomorphism such that
  \begin{equation*}
    \fr^*\omega_{I,-}=\omega_{J,+}, \quad
    \fr^*\omega_{J,-}=\omega_{I,+} \qandq
    \fr^*\omega_{K,-}=-\omega_{K,+}.
  \end{equation*}
\end{definition}

Given a matched pair of framed building blocks $(Z_\pm,\Sigma_\pm,\bomega_\pm;\fr)$, the twisted connected sum construction produces a simply-connected compact $7$--manifold $Y$ together with a family of torsion-free $\Gtwo$--structures $\set{\phi_T : T \gg 1 }$ by gluing truncations of $Y_\pm$ along their boundaries via interchanging the circle factors and $\fr$.

Sá Earp \cite{SaEarp2011} proved that given a holomorphic vector bundle $\sE$ on a building block with $\sE|_\Sigma$ $\mu$--stable,\footnote{Recall
  that a holomorphic bundle $\sE$ on a compact Kähler $n$--fold $(X,\omega)$ is \defined{$\mu$--(semi)stable} if for each torsion-free coherent subsheaf $\sF \subset \sE$ with $0 < \rk \sF < \rk \sE$ we have ($\mu(\sF) \leq \mu(\sE)$) $\mu(\sF) < \mu(\sE)$.  Here $\mu(\sE) \coloneq \<c_1(\sE)\cup [\omega]^{n-1},[X]\>/\rk \sE$ is the \defined{slope} of $\sE$ (and similarly for $\sF$).}
the smooth vector bundle underlying $\sE|_V$ can be equipped with a Hermitian Yang--Mills connection which is asymptotic at infinity to the anti-self-dual connection $A_\infty$ inducing the holomorphic structure on $\sE|_\Sigma$ \cite{Donaldson1985}.
Building on this, Sá Earp and the author \cite{SaEarp2013} developed a method for constructing $\Gtwo$--instantons over twisted connected sums provided a pair $\sE_\pm$ of such bundles and a lift $\bar\fr \co \sE_+|_{\Sigma_+} \to \sE_-|_{\Sigma_-}$ of the hyperkähler rotation $\fr$, which pulls back $A_{\infty,-}$ to $A_{\infty,+}$ (and assuming certain transversality conditions).
The following is a \emph{very special} case of the main result of \cite{SaEarp2013}.

\begin{theorem}
  \label{Thm_RigidExample}
  Let $(Z_\pm,\Sigma_\pm,\bomega_\pm;\fr)$ be a matched pair of framed building blocks.
  Denote by $Y$ the compact $7$--manifold and by $\set{ \phi_T : T \gg 1 }$ the family of torsion-free $\Gtwo$--structures obtained from the twisted connected sum construction.
  Let $\sE_\pm\to Z_\pm$ be a pair of rank $r$ holomorphic vector bundles such that the following hold:
  \begin{itemize}
  \item
    $c_1(\sE_+|_{\Sigma_+}) = \fr^*c_1(\sE_-|_{\Sigma_-})$ and $c_2(\sE_+|_{\Sigma_+}) = \fr^*c_2(\sE_-|_{\Sigma_-})$.
  \item
    $\sE_\pm|_{\Sigma_\pm}$ is $\mu$--stable with respect to $\omega_{I,\pm}$ and spherical, i.e.,
    \begin{equation*}
       H^*(\Sigma_\pm,\sEnd_0(\sE_\pm|_{\Sigma_\pm})) = 0.
    \end{equation*}
  \item
    $\sE_\pm$ is infinitesimally rigid:
    \begin{equation}
      \label{Eq_NoDeformations}
      H^1(Z_\pm,\sEnd_0(\sE_\pm))=0.
    \end{equation}
  \end{itemize}
  Then there exists a $\U(r)$--bundle $E$ over $Y$ with
  \begin{equation}
    \label{Eq_ChernClassesOfE}
    c_1(E) = \Upsilon(c_1(\sE_+),c_1(\sE_-)) \qandq
    c_2(E) = \Upsilon(c_2(\sE_+),c_2(\sE_-))
  \end{equation}
  and a family of connections $\{ A_T : T \gg 1\}$ on the associated $\PU(r)$--bundle with $A_T$ being an irreducible unobstructed $\Gtwo$--instanton over $(Y,\phi_T)$.
\end{theorem}

\begin{remark}
  The map 
  \textls[-15]{\begin{equation*}
    \Upsilon\co\set{ ([\alpha_+],[\alpha_-]) \in  H^\ev(Z_+)\times H^\ev(Z_-): [\alpha_+]|_{\Sigma_+} = \fr^*([\alpha_-]|_{\Sigma_-}) } \to H^\ev(Y)
  \end{equation*}
  is} 
  the natural patching map denoted by $Y$ in \cite[Definition 4.15]{Corti2012a}.
\end{remark}

Let $\res_\pm \co H^2(Z_\pm) \to H^2(\Sigma_\pm)$ denote the restriction maps associated with the inclusions $\Sigma_\pm \subset Z_\pm$ and set
\begin{equation*}
  N_\pm \coloneq \im \res_\pm.
\end{equation*}
If $\sE_\pm|_{\Sigma_\pm}$ is spherical, then $c_1(\sE_\pm|_{\Sigma_\pm})$ must be non-zero;
hence, \autoref{Thm_RigidExample} cannot be applied in situations where $N_+ \cap \fr^*N_- = 0$.
In particular, this rules out all the examples in \cites{Kovalev2003,Kovalev2011} as well as the mass-produced examples in \cite{Corti2012a}.
This means that the list of currently known $\Gtwo$--manifolds to which \autoref{Thm_RigidExample} could potentially be applied is relatively short.
Moreover, it has proved rather difficult to find suitable $\sE_\pm$.

Crowley and Nordström \cite{Crowley2014} systematically studied twisted connected sums of building blocks arising from Fano $3$--folds with Picard number two;
in particular, those that arise from matchings with $N_+ \cap \fr^*N_- \neq 0$.
This note shows that for one such twisted connected sum the hypotheses of \autoref{Thm_RigidExample} can be satisfied.

\begin{theorem}
  \label{Thm_Ex}
  There exists a twisted connected sum $Y$ of a pair of Fano type building blocks $(Z_\pm,\Sigma_\pm)$, arising from  \#13 and \#14 in Mori and Mukai's classification of Fano $3$--folds with Picard number two \cite[Table 2]{Mori1981}, admitting a pair of rank $2$ holomorphic vector bundles $\sE_\pm$ as required by \autoref{Thm_RigidExample}.
  In particular, each of the resulting twisted connected sums $(Y,\phi_T)$ with $T \gg 1$ carries an irreducible unobstructed $\Gtwo$--instanton on an $\SO(3)$--bundle.
\end{theorem}

\begin{remark}
  In earlier work \cite{Walpuski2011} the author constructed examples of irreducible unobstructed $\Gtwo$--instantons over $\Gtwo$--manifolds arising from Joyce's generalised Kummer construction \cites{Joyce1996,Joyce1996a}.
  To the author's best knowledge, \autoref{Thm_Ex} provides the first example of an irreducible unobstructed $\Gtwo$--instanton over a twisted connected sum.
\end{remark}

The method of proof relies mostly on certain arithmetic properties enjoyed by the twisted connected sum listed as \cite[Table 4, Line 16]{Crowley2014} by Crowley and Nordström.
A more abstract existence theorem is stated as \autoref{Thm_Abstract}.
It is an interesting question to ask whether there are any further twisted connected sums to which this result can be applied.

Finally, it should be pointed out that there is a very recent preprint by Menet, Nordström and Sá Earp \cite{Menet2015} in which they use the more general main result of \cite{SaEarp2013} to construct one $\Gtwo$--instanton.


\paragraph{Acknowledgements}
The author is grateful to Johannes Nordström for pointing out the $\Gtwo$--manifold constructed in \cite{Crowley2014} on which the above example lives;
moreover, he thanks the anonymous referees for thoughtful comments and suggestions.
The author gratefully acknowledges support from the Simons Center for Geometry and Physics, Stony Brook University at which part of the research for this paper was carried out.

\section{The twisted connected sum}
\label{Sec_TCS}

In this section we provide further details on Fano type building blocks, explain how to construct matching pairs of framed building blocks and describe the twisted connected sum mentioned in \autoref{Thm_Ex}.

\subsection{Building blocks of Fano type}
\label{Sec_FanoTypeBuildingBlocks}

If $W$ is a Fano $3$--fold, then according to Shokurov \cite{Shokurov1979} a general divisor $\Sigma \in |-K_W|$ is a smooth K3 surface.
Given a general pencil $|\Sigma_0,\Sigma_\infty| \subset |-K_W|$, blowing-up its base locus yields a smooth $3$--fold $Z$ together with a base-point free anti-canonical pencil spanned by the proper transforms of $\Sigma_0$ and $\Sigma_\infty$.
The resulting projective morphism $f\co Z \to \P^1$ makes $(Z,\Sigma_\infty)$ into a building block with
\begin{equation}
  \label{Eq_N}
  N \coloneq \im\(\res\co H^2(Z) \to H^2(\Sigma)\) \iso \Pic(W),
\end{equation}
see \cite[Proposition 6.42]{Kovalev2003} and \cite[Proposition 3.15]{Corti2012a}.

\citet[Theorem 7.5]{Moishezon1967} showed that if $-K_W$ is very ample for a very general\footnote{\label{Footnote_VeryGeneral}Here \defined{very general} means that the set of $\Sigma \in |-K_W|$ not satisfying the asserted condition is a countable union of complex analytic submanifolds of positive codimension in $\P H^0(-K_W)$.} $\Sigma \in |-K_W|$ we have  $\Pic(\Sigma) = \Pic(W)$.
Moreover, according to \citet[Proposition 2.14]{Kovalev2009} (see also, \citet[Corollary 2.10]{Voisin2007}) we can assume that $f\co Z \to \P^1$ is a \defined{rational double point (RDP) $K3$ fibration}, by which we mean that it has at only finitely many singular fibres and the singular fibres have only RDP singularities.
(In fact, Kovalev asserts that generically the singular fibres have only ordinary double points.)

\subsection{Matching building blocks}
\label{Sec_Matching}

Fix a lattice $L$ which is isomorphic to $(H^2(\Sigma),\cup)$ for $\Sigma$ a $K3$ surface.
Using the Torelli theorem and Yau's solution to the Calabi conjecture, \citet[section 6]{Corti2012a} showed that a set of framings of a pair of building blocks $Z_\pm$ together with a matching is equivalent (up to the action of $\O(L)$) to lattice isomorphisms $h_\pm \co L \to H^2(\Sigma_\pm)$ and an orthonormal triple $(k_+,k_-,k_0)$ of positive classes in $L_\R \coloneq L \otimes_\Z \R$ with $h_\pm(k_\pm)$ the restriction of a Kähler class on $Z_\pm$ and $\<k_\mp,\pm k_0\>$ the period point of $(\Sigma_\pm,h_\pm)$.
(The corresponding framings have $[\omega_{I,\pm}] = h_\pm(k_\pm)$ and the matching is such that $\fr^* = h_+\circ h_-^{-1}$.)
The following definition is useful to further simplify the matching problem.

\begin{definition}
  \label{Def_AmpGeneric}  
  Let $\cZ$ be a family of building blocks with constant $N$ and a fixed primitive isometric embedding $N \subset L$.
  Let $\Amp$ be an open subcone of the positive cone in $N_\R$.
  $\cZ$ is called \defined{$(N,\Amp)$--generic} if there exists a subset $U_\cZ\subset D_N \coloneq \set{ \Pi \in \P(N^\perp_\C) : \Pi\, \bar\Pi > 0 }$ with complement a countable union of complex analytic submanifolds of positive codimension and with the property that for any $\Pi\in U_\cZ$ and $k\in\Amp$ there exists a $(Z,\Sigma) \in \cZ$ and a marking $h\co L\to H^2(\Sigma)$ such that $\Pi$ is the period point of $(\Sigma,h)$ and $h(k)$ is the restriction to $\Sigma$ of a Kähler class on $Z$.
\end{definition}

This definition slightly deviates from \cite[Definition 6.17]{Corti2012a}.
There it is required that the complement of $U_\cZ$ is a \emph{locally finite} union of complex analytic submanifolds of positive codimension.
The above slightly weaker condition still suffices for the proof of the next proposition to carry over verbatim.

\begin{prop}[{\cite[Proposition 6.18]{Corti2012a}}]
  \label{Prop_Matching}
  Let $N_\pm\subset L$ be a pair of primitive sublattices of signature $(1,r_\pm-1)$ and let $\cZ_\pm$ be a pair of $(N_\pm,\Amp_\pm)$--generic families of building blocks.
  Suppose that $W \coloneq N_+ + N_-$ is an orthogonal pushout.\footnote{This
    means that $W_\R = W_{+,\R} \oplus W_{-,\R} \oplus (N_{+,\R} \cap N_{-,\R})$.
  }
  Set $T_\pm\coloneq N_\pm^\perp$ and $W_\pm \coloneq N_\pm \cap T_\mp$.
  If
  \begin{equation*}
    \Amp_\pm \cap W_\pm \neq \emptyset,
  \end{equation*}
  then there exist $(Z_\pm,\Sigma_\pm)\in\cZ_\pm$, markings $h_\pm\co L \to H^2(\Sigma_\pm)$ compatible with the given embeddings $N_\pm \subset L$ and an orthonormal triple $(k_+,k_-,k_0)$ of positive classes in $L_\R$ with:
  \begin{itemize}
  \item
    $k_\pm \in \Amp_\pm \cap W_{\pm,\R}$ and $k_0 \in W^\perp$,
  \item
    $h_\pm(k_\pm)$ the restriction of a Kähler class on $Z_\pm$, and
  \item
    $\<k_\mp,\pm k_0\>$ the period point of $(\Sigma_\pm,h_\pm)$.
  \end{itemize}
\end{prop}

If $\cZ$ is a family of building blocks arising from a full deformation type of Fano $3$--folds, then we can always find an open subcone $\Amp$ of the positive cone such that $\cZ$ is $(N,\Amp)$--generic \cite[Proposition 6.9]{Corti2012}.
(Also \autoref{Def_AmpGeneric} allows to slightly shrink $\cZ$ from a full deformation type by imposing very general conditions in the sense of \autoref{Footnote_VeryGeneral}.)
This reduces finding a matching of a pair such families of building blocks to the arithmetic problem of embedding $N_\pm$ into $L$ compatible with \autoref{Prop_Matching}.

\subsection{An example due to Crowley and Nordström}
\label{Sec_CNExample}

We will now describe the twisted connected sum found by \citet[Table 4, Line 16]{Crowley2014} which we referred to in \autoref{Thm_Ex}.

Consider the following pair of Fano $3$--folds:
\begin{itemize}
\item 
  Denote by $Q \subset \P^4$ a smooth quadric.
  Let $W_+ \to Q$ denote the blow-up of $Q$ in a degree $6$ genus $2$ curve \cite[Table 2, \#13]{Mori1981}.
\item
  Denote by $V_5$ a section of the Plücker-embedded Grassmannian $\Gr(2, 5) \subset \P^9$ by a subspace of codimension $3$.
  Let $W_- \to V_5$ denote the blow-up of $V_5$ in a elliptic curve that is the intersection of two hyperplane sections \cite[Table 2, \#14]{Mori1981}.
\end{itemize}
The anticanonical divisors $-K_{W_\pm}$ both are very ample.
To see this note that by according to \cite[Section 1]{Iskovskih1978} if $W$ is an index $r$ Fano $3$--fold and $-K_W$ is not very ample, then either $\abs{-\frac{1}{r}K_W}$ has a base point or $W$ is hyperelliptic.
According to \cite[Remarks preceding Table 12.3]{Iskovskih1999} neither is the case for the Fano $3$--folds under consideration;
see also  \cite[Theorem 2.1.16 and Theorem 2.4.5]{Iskovskih1999}.

$W_\pm$ both have Picard number $2$ with $\Pic(W_\pm)$ generated by $H_\pm$, the pullback of a generator of $\Pic(Q)$ and $\Pic(V_5)$ respectively, and the exceptional divisor $E_\pm$.
With respect to the bases $(H_\pm,E_\pm)$ the intersection forms on $N_\pm = \Pic(W_\pm)$, see \eqref{Eq_N}, can be written as
\begin{equation*} 
  \begin{pmatrix}
    6 & 6 \\
    6 & 2
  \end{pmatrix}
  \qandq
  \begin{pmatrix}
    10 & 5 \\
    5 & 0
  \end{pmatrix}
\end{equation*}
respectively.

$N_\pm$ can be thought of as the overlattices $\Z^2 + \frac15(3,-1)\Z$ and $\Z^2 + \frac16(1,1)\Z$ of $\Z^2$, generated by
\begin{equation}
  \label{Eq_AB}
  \begin{split}
    A_+=3H_+-E_+ \qandq B_+=4H_+-3E_+, \\ \andq 
    A_-= 3H_- - 2E_- \qandq B_- = 3H_- - 4E_-,
  \end{split}
\end{equation}
with intersection forms
\begin{equation*}
  \begin{pmatrix}
    20 & 0  \\
    0 & -30
  \end{pmatrix}
  \qandq
  \begin{pmatrix}
    30 & 0 \\
    0 & -30
  \end{pmatrix}
\end{equation*}
respectively.
The overlattice $W \coloneq \Z^3 + \frac15(3,0,-1)\Z + \frac16(0,
1,1)\Z$ of $\Z^3$ with intersection form
\begin{equation*}
  \begin{pmatrix}
    20 & 0 & 0 \\
    0 & 30 & 0 \\
    0 & 0 & -30
  \end{pmatrix}
\end{equation*}
is an orthogonal pushout of $N_\pm$ along $R = N_+ \cap N_- = (-30)$.

By Nikulin \cite[Theorem 1.12.4 and Corollary 1.12.3]{Nikulin1979} the lattice $W$ (and thus also $N_\pm$) can be embedded primitively into $L$.
Since we can choose $\Amp_\pm$ such that $\Amp_\pm \cap W_\pm$ is spanned by $A_\pm$, \autoref{Prop_Matching} yields matching data with $k_\pm$ a multiple of $A_\pm$ for a pair of building blocks $(Z_\pm,\Sigma_\pm)$ of Fano type arising from $W_\pm$.
Moreover, the resulting matching $\fr$ is such that $B_+ = \fr^*B_-$ (which generates $N_+ \cap \fr^*N_-$).
By the discussion at the end of \autoref{Sec_FanoTypeBuildingBlocks} we may assume that for all but countably many $b \in \P^1$ the fibre $\Sigma_{\pm,b} \coloneq f_\pm^{-1}(b)$ satisfies $\Pic(\Sigma_{\pm,b}) = N_\pm$;
in particular, we may assume that this holds for $\Sigma_\pm = \Sigma_{\pm,\infty}$.
Moreover, we may assume that $f_\pm\co Z_\pm \to \P^1$ is an RDP $K3$ fibration.


\section{Bundles on the building blocks}
\label{Sec_Bundles}

We will now construct holomorphic vector bundles $\sE_\pm$ over the building blocks $Z_\pm$ such that the hypotheses of \autoref{Thm_RigidExample} are satisfied.

The following theorem provides a spherical $\mu$--semistable vector bundle $\sE_{\pm,b}$ with
\begin{equation}
  \label{Eq_ChosenChernClasses}
  \rk \sE_{\pm,b} = 2, \quad
  c_1(\sE_{\pm,b}) = B_\pm \qandq
  c_2(\sE_{\pm,b}) = -6
\end{equation}
with $B_\pm$ as in \eqref{Eq_AB} on each non-singular fibre $\Sigma_{\pm,b} \coloneq f_\pm^{-1}(b)$.

\begin{theorem}[{\citet[Theorem~2.1]{Kuleshov1990}}]
  \label{Thm_Kuleshov}
  Let $(\Sigma, A)$ be a polarised smooth $K3$ surface.
  If $(r,c_1,c_2) \in \N\times H^{1,1}(\Sigma,\Z)\times\Z$ are such that
  \begin{equation}
    \label{Eq_Arithmetic}
    2rc_2 - (r-1) c_1^2 - 2(r^2-1) = 0,%
    \footnote{%
      Here and in the following, for $x \in H^2(\Sigma)$, we write $x^2 \in \Z$ to denote $x \cup x \in H^4(\Sigma) \iso \Z$.}
  \end{equation}
  then there exists a spherical $\mu$--semistable vector bundle $\sE$ on $\Sigma$ with
  \begin{equation*}
    \rk\sE = r, \quad c_1(\sE) = c_1 \qandq c_2(\sE) = c_2.
  \end{equation*}
\end{theorem}  

\begin{remark}
  By Hirzebruch--Riemann--Roch, \eqref{Eq_Arithmetic} is equivalent to $\chi(\sEnd_0(\sE)) = 0$, a necessary condition for $\sE$ to be spherical.
\end{remark}

Set 
\begin{equation*}
  U_\pm \coloneq \set{ b \in \P^1 : \Sigma_{\pm,b}~\text{is non-singular and}~\Pic(\Sigma_{\pm,b}) \iso N_\pm }.
\end{equation*}
Since  $A_\pm^\perp \subset N_\pm$ is generated by $B_\pm$ and $B_\pm^2 = -30 < -6$, for $b \in U_\pm$ the following guarantees that $\sE_{\pm,b}$ is indeed $\mu$--stable (and thus stable%
\footnote{%
  Recall
  that a torsion-free coherent sheaf $\sE$ on a projective variety $(X,\sO(1))$ is called \defined{(semi)stable} if for each torsion-free coherent subsheaf $\sF \subset \sE$ with $0 < \rk \sF < \rk \sE$ we have ($p_\sF \leq p_\sE$) $p_\sF < p_\sE$.
  Here $p_\sE$ denotes the \defined{reduced Hilbert polynomial} of $\sE$, the unique polynomial satisfying $p_\sE(m) = \chi(\sE\otimes\sO(m))/\rk \sE$ for all $m \in \Z$, and we compare polynomials using the lexicographical order of their coefficients.

  The notions of $\mu$--stability and stability are closely related in case $(X,\sO(1))$ is smooth (and thus Kähler): because $p_\sE(m) = \deg \sO(1)/n! \cdot m^n + (\mu(\sE)+ \frac12\deg(K_X))/(n-1)! \cdot m^{n-1} + \cdots$, $\mu$--stable implies stable (and semistable implies $\mu$--semistable).
}).

\begin{prop}
  \label{Prop_Stability}
  In the situation of \autoref{Thm_Kuleshov}, if the divisibilities of $r$ and $c_1$ are coprime and for all non-zero $x \in H^{1,1}(\Sigma,\Z)$ perpendicular to $c_1(A)$ we have
  \begin{equation}
    \label{Eq_Squares}
    x^2 < -\frac{r^2(r^2-1)}{2},
  \end{equation}
  then $\sE$ is $\mu$--stable.
\end{prop}

\begin{proof}
  Suppose $\sF$ were a destabilising sheaf, i.e., a torsion-free subsheaf $\sF \subset \sE$ with $0 < \rk \sF < \rk \sE$ and $\mu(\sF) = \mu(\sE)$.
  Since $c_1(\sE)c_1(A) = \rk\sE\cdot\mu(\sE)$ (and similarly for $\sF$), $x \coloneq \rk\sE\cdot c_1(\sF) - \rk\sF\cdot c_1(\sE) \in c_1(A)^\perp$.
  The discriminant of $\sE$ is
  \begin{equation*}
    \Delta(\sE) \coloneq 2\rk \sE \cdot c_2(\sE) - (\rk \sE - 1)c_1(\sE)^2 = 2(r^2-1)
  \end{equation*}
  by \eqref{Eq_Arithmetic}.
  According to \cite[Theorem 4.C.3]{Huybrechts2010} we must have either
  \begin{equation*}
    -\frac{(\rk \sE)^2}{4} \Delta(\sE) \leq x^2,
  \end{equation*}
  which violates \eqref{Eq_Squares}, or $x = 0$.

  The latter, however, implies
  \begin{equation*}
    \rk \sE\cdot c_1(\sF) = \rk \sF\cdot c_1(\sE),
  \end{equation*}
  which is impossible because the divisibilities of $\rk \sE$ and $c_1(\sE)$ are coprime.
\end{proof}

As a consequence of this and the following, for $b \in U_\pm$ the moduli space of semistable bundles on $\Sigma_{\pm,b}$ satisfying \eqref{Eq_ChosenChernClasses} is a reduced point.

\begin{theorem}[{Mukai \cite[Theorem 6.1.6]{Huybrechts2010}}]
  \label{Thm_Mukai}
  Let $(\Sigma, A)$ be a polarised smooth $K3$ surface.
  Suppose that $\sE$ is a stable sheaf satisfying \eqref{Eq_Arithmetic} with $r=\rk\sE$, $c_1=c_1(\sE)$ and $c_2=c_2(\sE)$.
  Then $\sE$ is locally free and any other semistable sheaf satisfying the same condition must be isomorphic to $\sE$.
\end{theorem}

If we were able construct holomorphic vector bundles $\sE_\pm$ on $Z_\pm$ whose restrictions to the fibres $\Sigma_{\pm,b}$ with $b \in U_\pm$ agree with $\sE_{\pm,b}$ and which satisfy \eqref{Eq_NoDeformations}, then we could apply \autoref{Thm_RigidExample} and the proof of \autoref{Thm_Ex} would be complete.
To see this, note that $\infty \in U_\pm$ and thus $\sE_\pm|_{\Sigma_{\pm,\infty}}$ have the same rank, their characteristic classes are identified by $\fr^*$ (since $\fr^*B_- = B_+$ by construction) and both are $\mu$--stable.
The construction of $\sE_\pm$ is achieved using the following tool.
(Note that $\frac12 B_\pm^2 + 6 = -9$; hence, \eqref{Eq_SemisimpleImpliesSimple} holds in our situation in view of \eqref{Eq_ChosenChernClasses}.)

\begin{prop}
  \label{Prop_SphericalBundlesK3Fibrations}
  Let $f\co Z \to B$ be RDP $K3$ fibration from a projective $3$--fold $Z$ to a smooth curve $B$ and set $S \coloneq \set{ b \in B : \Sigma_b \coloneq f^{-1}(b) ~\text{is singular} }$.
  Let $(r,c_1,c_2) \in \N \times \im(\res\co H^2(Z) \to H^2(\Sigma_b)) \times \Z$ for some $b \notin S$ be such that \eqref{Eq_Arithmetic} holds and
  \begin{equation}
    \label{Eq_SemisimpleImpliesSimple}
    \gcd\(r, \frac12 c_1^2 - c_2\) = 1.
  \end{equation}
  Suppose that there is a set $U\subset B\setminus S$ whose complement is countable and for each $b \in U$ the moduli space $M_b$ of semistable bundles $\sE_b$ on $\Sigma_b$ with
  \begin{equation}
    \label{Eq_ChernClasses}
    \rk \sE_b = r, \quad c_1(\sE_b) = c_1 \qandq c_2(\sE_b) = c_2 
  \end{equation}
  consists of a single reduced point: $M_b = \set{[\sE_b]}$.
  Then there exists a holomorphic vector bundle $\sE$ over $Z$ such that, for all $b \in U$, $\sE|_{\Sigma_b} \iso \sE_b$.
  $\sE$ is spherical, i.e., $H^*(\sEnd_0(\sE)) = 0$ and unique up to twisting by a line bundle pulled-back from $B$.
\end{prop}

\begin{remark}
  Note that by Hirzebruch--Riemann--Roch $\chi(\sE_b) = \frac12 c_1^2 - c_2 +2\rk \sE_b$, so \eqref{Eq_SemisimpleImpliesSimple} is asking that $\rk \sE_b$ and $\chi(\sE_b)$ be coprime.
\end{remark}

This result is essentially contained in Thomas' work on sheaves on $K3$ fibrations~\cite[Theorem 4.5]{Thomas2000}.
Its proof heavily relies on the following generalisation of \autoref{Thm_Mukai}.

\begin{theorem}[{\citet[Proof of Theorem 4.5]{Thomas2000}}]
  \label{Thm_Thomas}
  Let $(\Sigma,A)$ be a polarised $K3$ surface with at worst RDP singularities.
  If $\sE$ is a stable coherent sheaf on $\Sigma$ with $\chi(\sEnd_0(\sE)) = 0$, then $\sE$ is locally free.
\end{theorem}

We also use the following simple observation.
\begin{prop}
  \label{Prop_SemisimpleImpliesSimple}
  If $\sE$ is a semistable sheaf with $\rk\sE$ and $\chi(\sE)$ coprime, then $\sE$ is stable.
\end{prop}

\begin{proof}
  If $\sE$ is destabilised by $\sF \subset \sE$ with $0 < \rk \sF < \rk \sE$, then $p_\sF = p_\sE$.
  In particular, evaluating at $m = 0$ we have
  \begin{equation*}
    \rk \sE \cdot \chi(\sF) = \rk \sF \cdot \chi(\sE).
  \end{equation*}
  This contradicts $\rk\sE$ and $\chi(\sE)$ being coprime.
\end{proof}

\begin{proof}[Proof of \autoref{Prop_SphericalBundlesK3Fibrations}]
  Consider the moduli functor $\underline \bM \co \Sch_B^{\rm op} \to \Set$ which assigns to a $B$--scheme $U$ the set
  \begin{equation*}
    \underline \bM(U) \coloneq \set{ \sE ~\text{a coherent sheaf over}~ Z \times_B U ~\text{satisfying}~ (\clubsuit) }/\sim.
  \end{equation*}
  Here $(\clubsuit)$ means that $\sE$ is flat over $U$, for each $b \in U$, $\sE\otimes_{\sO_U} k(b)$ is semistable and its Hilbert polynomial $P$ agrees with that of a sheaf on a smooth fibre with characteristic classes given by \eqref{Eq_ChernClasses}.
  We write $\sE \sim \sF$ if and only if there exists a line bundle $\sL$ over $U$ such that $\sE$ and $\sF\otimes\sL$ are $S$--equivalent;
  cf.~\citet[p.~561]{Maruyama1978} and \citet[Section 4.1]{Huybrechts2010}.

  $\underline\bM$ is universally corepresented by a proper and separated $B$--scheme $\bM$, i.e., the moduli problem has a proper and separated coarse moduli space, see \citet[Section 1]{Simpson1994}.
  The fibre of $\bM$ over $b \in B$ is the coarse moduli space of semistable sheaves on $\Sigma_b$ with Hilbert polynomial $P$.

  Denote by $M$ the component of $\bM$ whose fibres over $B\setminus S$ are the coarse moduli space $M_b$ of semistable sheaves $\sE$ on $\Sigma_b$ satisfying \eqref{Eq_ChernClasses}.
  By assumption, for each $b \in U$, $M_b$ consists of a single reduced point $[\sE_b]$.
  By \eqref{Eq_SemisimpleImpliesSimple} and \autoref{Prop_SemisimpleImpliesSimple}, $\sE_b$ is stable and, hence, spherical because $\chi(\sEnd_0(\sE_b)) = 0$.
  By \autoref{Thm_Mukai} it is locally free.
  Using deformation theory, see, e.g., \citet[Section 7]{Hartshorne2010}, one can show that $M \to B$ is surjective onto a open neighbourhood of each $b \in U$ and thus to all of $B$, since it is proper.
  Using  \eqref{Eq_SemisimpleImpliesSimple} and \autoref{Prop_SemisimpleImpliesSimple} as well as \autoref{Thm_Mukai} again we see that for each $b \notin S$ the fibre $M_b$ is a reduced point.
  Since $M$ is separated, it follows that $M = B$.

  By \cite[Corollary 4.6.7]{Huybrechts2010}, \eqref{Eq_SemisimpleImpliesSimple} guarantees the existence a universal sheaf $\sE$ on $Z \times_B M = Z$.%
  \footnote{%
    Strictly speaking, the quoted result only provides the universal sheaf over $f^{-1}(B\setminus S)$;
    however, the argument of \cite[second paragraph in the proof of Theorem 4.5]{Thomas2000} shows why the argument works uniformly on $B$.
    Alternatively, the existence of the universal sheaf can be deduced from \cite[Theorem 1.21]{Simpson1994} and Tsen's theorem $H^2_{\rm et}(B,\sO^*) = 0$.}
  By flatness, for each $b \in B$, $\chi(\sE|_{\Sigma_b}) = \frac12 c_1^2 - c_2 + 2r$ and $\chi(\sEnd_0(\sE|_{\Sigma_b})) = 0$.
  From \eqref{Eq_SemisimpleImpliesSimple}, \autoref{Prop_SemisimpleImpliesSimple} and \autoref{Thm_Thomas} (resp.~\autoref{Thm_Mukai}) it follows that $\sE|_{\Sigma_b}$ is locally free and spherical for arbitrary $b \in B$.
  Therefore, $\sE$ is also locally free by \cite[Lemma 1.27]{Simpson1994} and spherical by Grothendieck's spectral sequence.

  The asserted uniqueness property follows from the fact that $\sE$ is a universal sheaf and the definition of the moduli functor.
\end{proof}

This completes the construction of the bundles $\sE_\pm$ and thus the proof of \autoref{Thm_Ex}.
Clearly, the above argument also proves the following more abstract result.

\begin{theorem}
  \label{Thm_Abstract}
  Let $(Z_\pm,\Sigma_\pm,\bomega_\pm;\fr)$ be a matched pair of framed building blocks.
  Suppose that $f_\pm \co Z_\pm \to \P^1$ are RDP $K3$ fibrations and that for all but countably many $b \in \P^1$ we have
  \begin{equation*}
    \Pic(f_\pm^{-1}(b)) \iso N_\pm \coloneq \im\(\res_\pm\co H^2(Z_\pm) \to H^2(\Sigma_\pm)\).
  \end{equation*}
  Suppose there exists a $(r,c_1,c_2) \in \N \times (N_+ \cap \fr^*N_-) \times \Z$ such that
  \begin{equation*}
    2rc_2 - (r-1) c_1^2 - 2(r^2-1) = 0    
  \end{equation*}
  and
  \begin{equation*}
    \gcd\(r,\frac12 c_1^2 - c_2\) = 1.
  \end{equation*}
  If $[\omega_{I,\pm}] \in H^2(\Sigma_\pm,\Q)$ and for all non-zero $x \in [\omega_{I,\pm}]^\perp \subset N_\pm$ we have
  \begin{equation*}
    x^2 < -\frac{r^2(r^2-1)}{2},
  \end{equation*}
  then there exists rank $r$ holomorphic vector bundles $\sE_\pm$ on $Z_\pm$ with
  \begin{equation*}
    c_1(\sE_+|_{\Sigma_+}) = \fr^*c_1(\sE_-|_{\Sigma_-}) = c_1 \qandq
    c_2(\sE_+|_{\Sigma_+}) = \fr^*c_2(\sE_-|_{\Sigma_-}) = c_2
  \end{equation*}
  satisfying the hypotheses of \autoref{Thm_RigidExample}.
\end{theorem}


\printreferences

\end{document}
